\documentclass[pdftex]{amsart}
\usepackage[ansinew]{inputenc}
\usepackage[english]{babel}
\usepackage{amssymb}
\usepackage{amsmath}
\usepackage{amsthm}
\usepackage[pdftex]{color, graphicx}

\newtheorem{thm}{Theorem}
\newtheorem{lem}[thm]{Lemma}
\newtheorem{cor}[thm]{Corollary}
\newtheorem{prop}[thm]{Proposition}
\theoremstyle{definition}
\newtheorem{defn}{Definition}

\renewcommand{\>}{\rangle}

\newcommand{\C}{\mathcal C}

\newcommand{\F}{\mathcal F}
\newcommand{\N}{\mathbb N}
\newcommand{\Z}{\mathbb Z}
\newcommand{\Syl}{\mathrm{Syl}}
\newcommand{\Br}{\mathrm{Br}}
\newcommand{\Aut}{\mathrm{Aut}}

\title{Block fusion systems of the alternating groups}

\date{\today}

\author{Martin Wedel Jacobsen}
\address{Institut for Matematiske Fag\\
  Universitetsparken 5\\
  DK--2100 København}
\email{martinw@math.ku.dk}
\urladdr{}

\thanks{Supported by the Danish National Research Foundation (DNRF)
through the Centre for Symmetry and Deformation.}

\begin{document}

\begin{abstract}
We describe a purely group-theoretic condition on an element $g$ of
a finite group $G$ which implies that $g$ has coefficient zero in
every central idempotent element of the group ring $RG$, provided
that $R$ is a ring of prime characteristic. We use this condition to
prove that the fusion system associated to a block of an alternating
group is always isomorphic to the group fusion system of an
alternating group.
\end{abstract}

\maketitle
\tableofcontents

\section{Introduction}

In the study of modular representations, the group ring does not
factor into a product of matrix rings, as is the case with complex
representations. Instead, one can assign to each irreducible factor
a fusion system which provides some information about the structure
of the factor. It is an open problem whether these block fusion
systems always occur as group fusion systems as well. A few cases
are known: the block fusion systems of symmetric groups are always
group fusion systems of symmetric groups, and a similar result holds
for the general linear groups over finite fields (see \cite[Theorem
7.2 and Remarks 7.4]{Kes07}).

In this paper, we provide a new proof of the result on symmetric
groups (Theorem \ref{sym-fusion}), and we extend this proof to cover
the alternating groups. We obtain that a block fusion system of an
alternating group is always isomorphic to a group fusion system of
an alternating group (Corollary \ref{alt-fusion-cor}).

In Section \ref{sec-prelim}, we review the definition of block
fusion systems and recall a few of their properties. We also prove a
number of simple lemmas that will be needed later. In Section
\ref{sec-idem}, we consider central idempotents of the group ring
$RG$, where $R$ is any ring of prime characteristic $p$. We describe
a group-theoretic condition on an element $g$ of $G$ that implies
that $g$ has coefficient zero in all these central idempotents. In
Section \ref{sec-symalt}, we apply this condition to the symmetric
and the alternating groups in order to derive restrictions on the
possible defect groups. We then derive some properties of the
centric subgroups of these defect groups and use these properties to
determine the possible fusion systems.

Notation: Throughout this paper, we work in the group ring $RG$
where $R$ is a ring of characteristic $p$ and $G$ is a finite group.
When $g$ is an element of $G$, $[g]$ is the conjugacy class of $g$,
and when $\C$ is a conjugacy class of $G$, $\Sigma\C$ is the sum of
the elements in $\C$, considered as an element of $RG$.

An element of $G$ is called $p$-regular if it has order not
divisible by $p$, and it is called a $p$-element if its order is a
power of $p$. For any element $g$ of $G$, $g_p$ and $g_{p'}$ denote
the $p$-part and the $p'$-part of $g$ (see Lemma
\ref{grp-elem-part}). The notation $\prod_{i=a}^b f(i)$ always means
the product $f(a)\cdot f(a+1) \cdots f(b)$.

For a finite set $M$, $S_M$ and $A_M$ are the symmetric and
alternating groups on $M$, and for a natural number $n$, $S_n$ and
$A_n$ are the symmetric and alternating groups on $\{1, \ldots,
n\}$. $C_n$ is the cyclic group of order $n$. When $P$ is a subgroup
of $G$, $\Aut_G(P)$ is the group of automorphisms of $P$ that arise
from conjugation with an element of $G$. Likewise, when $\F$ is a
fusion system on a group containing $P$, $\Aut_\F(P)$ is the
automorphism group of $P$ in $\F$. When $S$ is a Sylow $p$-subgroup
of $G$, $\F_S(G)$ is the fusion system on $S$ generated by $G$.

\bigskip

I would like to thank my advisor, Jesper Michael Møller, for
introducing me to this problem and for providing invaluable guidance
throughout my work on it.

\section{Preliminaries}\label{sec-prelim}

We review the definition and basic properties of blocks of a finite
group, Brauer pairs associated to a block, and the fusion system
associated to a block. For more detailed background, we refer to
\cite{AKO} and \cite{The95}.

Let $k$ be an algebraically closed field of characteristic $p$ and
let $G$ be a finite group. The group ring $kG$ may be decomposed as
a direct product of $k$-algebras; each primitive factor is called a
block of $G$. The unit element of a block is called the block
idempotent. It is a primitive central idempotent of $kG$. We record
a few basic facts about central idempotents:

\begin{prop}
The product of any two distinct block idempotents of $kG$ is 0. Any
central idempotent of $kG$ is the sum of distinct block idempotents.
The sum of all the block idempotents of $kG$ is 1.
\end{prop}

\begin{proof}
Apply \cite[Corollary 4.2]{The95} to $Z(kG)$.
\end{proof}

\begin{lem}
Let $A$ be a finite-dimensional algebra over $k$, and let $\varphi:
A \to B$ be a surjective algebra homomorphism. The image of a
primitive central idempotent of $A$ is a central idempotent of $B$
(possibly zero), and every central idempotent of $B$ that lies in
$\varphi(Z(A))$ can be lifted to a central idempotent of $A$.
\end{lem}

\begin{proof}
Since $\varphi$ is surjective, $\varphi(Z(A))$ is contained in
$Z(B)$. This covers the first part. For the second part, apply
\cite[Theorem 3.2(b)]{The95} to $Z(A)$.
\end{proof}

We recall the definition of block fusion systems by means of Brauer
pairs. More detailed information can be found in \cite[Part
IV]{AKO}.

\begin{defn}
A Brauer pair of $G$ is a pair $(P,e)$ where $P$ is a $p$-subgroup
of $G$ and $e$ is a block idempotent of $kC_G(P)$.
\end{defn}

\begin{defn}
Let $Q \unlhd P$ be $p$-subgroups of $G$. The Brauer homomorphism
$\Br_{P/Q}: (kC_G(Q))^P \to kC_G(P)$ is given by mapping $\sum_{g
\in C_G(Q)} c_gg$ to $\sum_{g \in C_G(P)} c_gg$ (note the change in
the range of the sum). It is a $k$-algebra homomorphism.
\end{defn}

A Brauer pair whose first part is $P$ is also referred to as a
Brauer pair at $P$. Note that there is an obvious one-to-one
correspondence between Brauer pairs at the trivial group and block
idempotents of $G$. When $Q$ is the trivial group, the Brauer
homomorphism is denoted $\Br_P$.

\begin{defn}
Inclusion of Brauer pairs is defined as follows. We write $(Q,f)
\unlhd (P,e)$ if $Q \unlhd P$, $f \in (kC_G(Q))^P$, and
$\Br_{P/Q}(f)e = e$. We define the relation $\leq$ between Brauer
pairs to be the transitive closure of the relation $\unlhd$.
Conjugation of Brauer pairs is defined by $(P,e)^g = (P^g, e^g)$. A
Brauer pair $(P,e)$ is said to be associated to a block idempotent
$b$ of $G$ if $(1,b) \leq (P,e)$.
\end{defn}

The condition $\Br_{P/Q}(f)e = e$ deserves further explanation.
Since $\Br_{P/Q}$ is surjective, $\Br_{P/Q}(f)$ is a central
idempotent of $kC_G(Q)$; it therefore decomposes as a sum of
distinct block idempotents. The condition $\Br_{P/Q}(f)e = e$ then
means that $e$ appears in this decomposition.

\begin{prop}
Given a Brauer pair $(P,e)$ and a subgroup $Q$ of $G$ with $Q \leq
P$, there is a unique block idempotent $f$ of $C_G(Q)$ such that
$(Q,f) \leq (P,e)$. If in addition $Q \unlhd P$, then $(Q,f) \unlhd
(P,e)$. In particular, every Brauer pair of $G$ is associated to a
unique block idempotent. Additionally, for a given block $b$, any
two maximal Brauer pairs associated to $b$ are conjugate.
\end{prop}

\begin{defn}
Let $b$ be a block of $G$. If $(P,e)$ is a maximal Brauer pair
associated to $b$, then $P$ is called a defect group of $b$. The
fusion system on $(P,e)$ associated to $b$ is a category defined as
follows. The objects are the subpairs of $(P,e)$; the maps from
$(Q,f)$ to $(Q',f')$ are the injective group homomorphisms $\varphi:
Q \to Q'$ with the property that there is a $g \in G$ such that
$\varphi(x) = x^g$ and $(Q,f)^g = (Q',f')$.
\end{defn}

\begin{prop}
The fusion system on $(P,e)$ associated to $b$ is a saturated fusion
system on $P$. Because maximal Brauer pairs associated to $b$ are
conjugate, the fusion system does not depend on the choice of
maximal Brauer pair.
\end{prop}

For the definition and properties of saturated fusion systems, we
refer to \cite[Part I]{AKO}. For this article, we will only need the
following two facts.

\begin{lem}
Let $P$ be a $p$-group, and let $\F$ be a saturated fusion system on
$P$. Then $\Aut_P(P)$ is a Sylow $p$-subgroup of $\Aut_\F(P)$.
\end{lem}

\begin{proof}
By \cite[Definition I.2.2]{AKO}, $P$ is $\F$-conjugate to a fully
automized subgroup of $P$, which must be $P$ itself, since no proper
subgroup of $P$ is isomorphic to $P$. The required property then
follows from the definition of fully automized.
\end{proof}

\begin{thm}[Alperin's fusion theorem, weak form]\label{alperin}
Let $P$ be a $p$-group, and let $\F$ be a saturated fusion system on
$P$. Then $\F$ is uniquely determined by the groups $\Aut_\F(Q)$
where $Q$ runs over the centric subgroups of $P$.
\end{thm}

\begin{proof}
By \cite[Theorem I.3.5]{AKO}, $\F$ is determined by the groups
$\Aut_\F(Q)$ where $Q$ runs over the $\F$-essential subgroups of
$P$, together with the group $\Aut_\F(P)$. By \cite[Proposition
I.3.3(a)]{AKO}, an $\F$-essential subgroup of $P$ is $\F$-centric in
$P$, so by definition it is also centric in $P$. Additionally, $P$
itself is also centric in $P$.
\end{proof}

We will make use of the following simple observation regarding
defect groups.

\begin{lem}\label{def-syl}
Let $e$ be a block idempotent of $G$, and let $P$ be a defect group
of $e$. Then $e$ contains an element $a \in G$ such that $P$ is a
Sylow $p$-subgroup of $C_G(a)$.
\end{lem}

\begin{proof}
Since $P$ is a defect group of $e$, we have $\Br_P(e) \neq 0$, so we
may choose an $a \in G$ with nonzero coefficient in $\Br_P(e)$. Then
$a$ also has nonzero coefficient in $e$, and we have $a \in C_G(P)$.
This obviously implies $P \subseteq C_G(a)$. Now let $S$ be a
$p$-subgroup of $G$ that properly contains $P$. Since $P$ is a
defect group of $e$ and all defect groups of $e$ are conjugate,
there can be no Brauer pairs at $S$ associated to $e$. Then
$\Br_S(e) = 0$, so we have $a \not\in C_G(S)$. This implies $S
\not\subseteq C_G(a)$, so $P$ is a maximal $p$-subgroup of $C_G(a)$.
It is then a Sylow $p$-subgroup of $C_G(a)$.
\end{proof}

We also collect a few group-theoretic lemmas we will need later.

\begin{lem}\label{grp-elem-part}
Let $G$ be a finite group and $p$ a prime. For any element $g$ of
$G$, there are unique elements $g_p$ and $g_{p'}$ of $G$, such that
$g_p$ is a $p$-element, $g_{p'}$ is $p$-regular, and $g = g_pg_{p'}
= g_{p'}g_p$.
\end{lem}

\begin{proof}
Let $g$ have order $p^km$ where $p \nmid m$, and let $x$ and $y$ be
integers such that $xm+yp^k = 1$. Set $g_p = g^{xm}$ and $g_{p'} =
g^{yp^k}$; then $g = g_pg_{p'} = g_{p'}g_p$. Further, $x$ and $p^k$
are coprime, so $g_p$ has order $p^k$; similarly, $g_{p'}$ has order
$m$ since $m$ and $y$ are coprime. This proves the existence of
$g_p$ and $g_{p'}$.

For uniqueness, suppose that we are given $g_p$ and $g_{p'}$ with
the desired properties. Then we have $g^{xm} = (g_pg_{p'})^{xm} =
g_p^{xm}g_{p'}^{xm} = g_p^{xm}$, since $g_p$ and $g_{p'}$ commute
and $g_{p'}$ has order $m$. As $g_p$ has order $p^k$, we get
$g_p^{xm} = g_p^{xm+yp^k} = g_p$. Thus we must have $g_p = g^{xm}$.
We similarly deduce $g_{p'} = g^{yp^k}$.
\end{proof}

The elements $g_p$ and $g_{p'}$ may be called the $p$-part and the
$p'$-part of $g$, respectively.

\begin{lem}\label{exp-reg}
Let $G$ be a finite group, and let $p$ be a prime. Then there exists
a number $n \in \N$ such that $g^{p^n} = g_{p'}$ for any $g \in G$.
\end{lem}

\begin{proof}
Write $|G| = p^k \cdot m$ with $m$ not divisible by $p$. Choose $n$
such that $p^n$ is congruent to 1 modulo $m$ and $n \geq k$. Then if
$g$ is a $p$-element, its order is a divisor of $p^k$, so $g^{p^n}$
is the identity. If $g$ is a $p$-regular element, its order is a
divisor of $m$, so $g^{p^n} = g^1 = g$. For an arbitrary $g \in G$,
we then have
\[
g^{p^n} = (g_p \cdot g_{p'})^{p^n} = g_p^{p^n}\cdot g_{p'}^{p^n} = g_{p'}
\]
\end{proof}

We call a number $q = p^n$ a \emph{$p$-regular exponent} of $G$ if
it satisfies the conditions of Lemma \ref{exp-reg}.

\begin{lem}\label{syl-konj}
Let $G$ be a finite group, $S$ a Sylow $p$-subgroup of $G$, and $a$
a $p$-element of $G$. Then the group $\<S^{a^k} \mid k \in \Z\>$
contains $a$.
\end{lem}

\begin{proof}
Let $H = \<S^{a^k} \mid k \in \Z\>$. Since $S \subseteq H \subseteq
G$ and $S$ is a Sylow $p$-subgroup of $G$, $S$ is also a Sylow
$p$-subgroup of $H$. Then $\Syl_p(H)$ is a subset of $\Syl_p(G)$.

Since conjugation by $a$ maps $S^{a^k}$ into $S^{a^{k+1}}$, $a$
normalizes $H$. Then the group $\<a\>$, whose order is a power of
$p$, acts on $H$ by conjugation. In particular, it acts on
$\Syl_p(H)$ by conjugation. Any orbit of this action has size
divisible by $p$ unless it consists of a single element, and since
$|\Syl_p(H)|$ is congruent to 1 modulo $p$, the action then has a
fixed point. That is, there is a group $T \in \Syl_p(H)$ such that
$T^a = T$. This implies that $\<T, a\>$ is a $p$-subgroup of $G$
containing $T$. Because $T$ is a Sylow $p$-subgroup of $G$, we must
then have $T = \<T, a\>$, and then $a \in T \subseteq H$.
\end{proof}

\begin{lem}\label{konj-prod}
Let $a$ and $b$ be elements of the group $G$, and let $n \in \N$.
Then $\prod_{i=0}^{n-1} a^{b^{-i}} = (ab)^n\cdot b^{-n}$.
\end{lem}

\begin{proof}
By induction. The statement is clear for $n = 1$, and for higher
$n$, we have
\begin{align*}
\prod_{i=0}^n a^{b^{-i}} & = \left(\prod_{i=0}^{n-1} a^{b^{-i}}\right) \cdot a^{b^{-n}}
= (ab)^n \cdot b^{-n} \cdot a^{b^{-n}} = (ab)^n \cdot a \cdot b^{-n} \\
& = (ab)^n \cdot ab \cdot b^{-(n+1)} = (ab)^{n+1} \cdot b^{-(n+1)}
\end{align*}
\end{proof}

\begin{lem}\label{cent-orb}
Let $G$ be a finite group that acts on a set $X$, let $H$ be a
subgroup of $G$, and let $O_1, \ldots, O_k$ be the orbits of the
action of $H$ on $X$. Then the action of $C_G(H)$ on $X$ induces an
action on the set $\{O_1, \ldots, O_k\}$ by $g(O_i) = \{ g(x) \mid x
\in O_i\}$.
\end{lem}

\begin{proof}
Let $O$ be an orbit of $X$ under $H$, let $x \in O$, and let $g \in
C_G(H)$. Then for any $h \in H$, we have $h(g(x)) = g(h(x))$, so the
orbit of $X$ containing $g(x)$ is precisely $g(O)$.
\end{proof}

\section{Central idempotents of the group ring}\label{sec-idem}

In this section, we analyze the central idempotents of $RG$, where
$R$ is a ring of characteristic $p$. Note that we do not assume that
$R$ is a field, or even a commutative ring. It is easily seen that
$Z(RG)$ consists of all elements of the form $\sum_{i=1}^k r_i \cdot
\Sigma\C_i$ where $r_i \in Z(R)$ and each $\C_i$ is a conjugacy
class of $G$.

\begin{lem}\label{exp-block}
Let $a \in G$, $n \in \N$, and suppose that for every conjugacy
class $\C$ in $G$, the coefficient of $a$ in $(\Sigma\C)^{p^n}$ is
zero. Then $a$ has coefficient zero in all central idempotents of
$RG$.
\end{lem}

\begin{proof}
Let $e$ be a central idempotent of $RG$, and write $e = \sum_{i=1}^k
r_i \cdot \Sigma\C_i$ for some elements $r_i \in Z(R)$ and conjugacy
classes $\C_i$ of $G$. Since $e$ is idempotent, we have $e^{p^n} =
e$. As $Z(RG)$ is commutative, we may apply Freshman's Dream, and we
find
\[
e = e^{p^n} = \left(\sum_{i=1}^k r_i \cdot \Sigma\C_i\right)^{p^n}
= \sum_{i=1}^k r_i^{p^n} \cdot (\Sigma\C_i)^{p^n}
\]
By assumption, $a$ has coefficient zero on the right hand side, so
it also has coefficient zero in $e$.
\end{proof}

When $q$ is a $p$-regular exponent, it turns out to be fairly simple
to describe $(\Sigma\C)^q$:

\begin{thm}\label{exp-coef}
Let $G$ be a finite group, $\C$ a conjugacy class of $G$, $a$ an
element of $G$, and $q$ a $p$-regular exponent of $G$. If $a$ is not
$p$-regular, then the coefficient of $a$ in $(\Sigma\C)^q$ is zero.
If $a$ is $p$-regular, let $S$ be a Sylow $p$-subgroup of $C_G(a)$,
and let $H$ be the set of all $p$-elements of $C_G(S)$. Then the
coefficient of $a$ in $(\Sigma\C)^q$ is equal to the number of
elements $h$ of $H$ such that $ah \in \C$.
\end{thm}

\begin{proof}
Write $\C^q$ for the set of $q$-tuples of elements of $\C$, and
define a map $\pi: \C^q \to G$ by $\pi(x_1, \ldots, x_q) =
\prod_{i=1}^q x_i$. Then clearly $(\Sigma\C)^q = \sum_{\alpha \in
\C^q} \pi(\alpha)$.

$G$ acts on $\C^q$ by $(x_1, \ldots, x_q)^g = (x_1^g, \ldots,
x_q^g)$, and this action clearly satisfies $\pi(\alpha^g) =
\pi(\alpha)^g$. We also define an action of $C_q$ on $\C^q$ as
follows. Fix a generator $\sigma$ of $C_q$, and define $(x_1, x_2,
\ldots, x_q)^\sigma = (x_2, \ldots, x_q, x_1)$. Then for $\alpha =
(x_1, x_2, \ldots, x_q)$, we find $\pi(\alpha^\sigma) =
\prod_{i=2}^q x_i \cdot x_1 = x_1^{-1} \cdot \prod_{i=1}^q x_i \cdot
x_1 = \pi(\alpha)^{x_1}$, so $\pi(\alpha^\sigma)$ and $\pi(\alpha)$
are conjugate in $G$. Additionally, this action commutes with the
action of $G$, so we have an action of $G \times C_q$ on $\C^q$.

Consider for a fixed $\alpha \in \C^q$ the set $M = \{\pi(\alpha^g)
\mid g \in G \times C_q\}$ where we count elements with
multiplicity. We partition $M$ into $q$ sets $M_0, M_1, \ldots,
M_{q-1}$ by defining $M_i = \{\pi(\alpha^{(g,\sigma^i)}) \mid g \in
G\}$. Since $\pi(\alpha^{(g,\sigma^i)}) = \pi(\alpha^{\sigma^i})^g$,
$M_i$ consists of $|C_G(\pi(\alpha^{\sigma^i}))|$ copies of
$[\pi(\alpha^{\sigma^i})]$. As $\pi(\alpha)$ and
$\pi(\alpha^{\sigma^i})$ are conjugate in $G$, we have
$|C_G(\pi(\alpha^{\sigma^i}))| = |C_G(\pi(\alpha))|$ and
$[\pi(\alpha^{\sigma^i})] = [\pi(\alpha)]$, so $M$ consists of $q
\cdot |C_G(\pi(\alpha))|$ copies of $[\pi(\alpha)]$. Now when $g$
runs over $G\times C_q$, $\alpha^g$ runs over $\alpha^{G\times C_q}$
exactly $|C_{G\times C_q}(\alpha)|$ times, so we find
\[
\sum_{\beta \in \alpha^{G\times C_q}} \pi(\beta)
= \frac{q\cdot |C_G(\pi(\alpha))|}{|C_{G\times C_q}(\alpha)|}
\cdot \Sigma[\pi(\alpha)]
\]

Now consider the map $\varphi: C_{G\times C_q}(\alpha) \to C_q$
defined as the restriction of the projection $G \times C_q \to C_q$.
The kernel of $\varphi$ is clearly $C_G(\alpha)$, and the image is a
subgroup of $C_q$ of order $p^k$, say. Then $|C_{G\times
C_q}(\alpha)| = p^k \cdot |C_G(\alpha)|$. Since $\alpha^g = \alpha$
implies $\pi(\alpha)^g = \pi(\alpha^g) = \pi(\alpha)$, we also have
$C_G(\alpha) \subseteq C_G(\pi(\alpha))$. Inserting this in the
above equation, we get
\[
\sum_{\beta \in \alpha^{G\times C_q}} \pi(\beta)
= \frac{q}{p^k} \cdot |C_G(\pi(\alpha)) : C_G(\alpha)|
\cdot \Sigma[\pi(\alpha)]
\]

If $\varphi$ is not surjective, $p^k$ is a smaller power of $p$ than
$q$, and then $\frac{q}{p^k}$ is divisible by $p$. Since we are
working in characteristic $p$, this immediately implies $\sum_{\beta
\in \alpha^{G\times C_q}} \pi(\beta) = 0$. Furthermore, the question
of whether $\varphi$ is surjective depends only on $\alpha^{G\times
C_q}$ since if $\alpha$ and $\alpha'$ lie in the same orbit, then
$C_{G\times C_q}(\alpha)$ and $C_{G\times C_q}(\alpha')$ are
conjugate in $G\times C_q$. We define $X \subseteq \C^q$ to consist
of those $\alpha \in \C^q$ for which $\varphi$ is surjective; this
is then a union of $G \times C_q$-orbits, and we have
\[
(\Sigma\C)^q = \sum_{\alpha\in\C^q} \pi(\alpha) = \sum_{\alpha \in
X} \pi(\alpha)
\]

Now let $\alpha \in X$ and write $\alpha = (x_1, x_2, \ldots, x_q)$.
Since the projection of $C_{G\times C_q}(\alpha)$ onto $C_q$ is
surjective, there exists an $h \in G$ such that $(h,\sigma) \in
C_{G\times C_q}(\alpha)$. We then have $(x_1, x_2, \ldots, x_q) =
(x_1, x_2, \ldots, x_q)^{(h,\sigma)} = (x_2^h, \ldots, x_q^h,
x_1^h)$. This shows that $x_{i+1}^h = x_i$ for $1 \leq i \leq q-1$;
by induction, we then obtain $x_i = x_1^{h^{1-i}}$ for $1 \leq i
\leq q$. This shows that $\alpha$ is determined by $x_1$ and $h$
alone. Further, the equation $x_1^h = x_q$ implies $x_1^{h^{-q}} =
x_1$, so that $x_1$ and $h^q$ commute.

We say that $(x,h)$ is a \emph{defining pair} for $\alpha$ if $x$
and $h^q$ commute and $\alpha = (x, x^{h^{-1}}, x^{h^{-2}}, \ldots,
x^{h^{1-q}})$. The above then shows that every $\alpha \in X$ has a
defining pair.

Let $\alpha \in X$ and let $(x,h)$ be a defining pair for $\alpha$.
Since $q$ is a $p$-regular exponent of $G$, we have $h^q = h_{p'}$.
Then $h_{p'}$ commutes with $x$, and for any $k \in \Z$, we have
$x^{h^{-k}} = x^{h_{p'}^{-k} \cdot h_p^{-k}} = x^{h_p^{-k}}$.
Additionally, $h_p^q$ commutes with $x$ since it is the identity
element, so $(x,h_p)$ is a defining pair for $\alpha$. Then every
$\alpha \in X$ has a defining pair $(x,h)$ where $h$ is a
$p$-element; for these pairs, the condition that $x$ and $h^q$
commute is redundant since $h^q$ is the identity. From now on, it
will be assumed that if $(x,h)$ is a defining pair, then $h$ is a
$p$-element.

Let again $\alpha \in X$ and let $(x,h)$ be a defining pair for
$\alpha$. Using Lemma \ref{konj-prod}, we find
\[
\pi(\alpha) = \prod_{i=0}^{q-1} x^{h^{-i}} = (xh)^q \cdot h^{-q}
= (xh)^q = (xh)_{p'}
\]
In particular, $\pi(\alpha)$ is a $p$-regular element. This proves
the first part of the theorem.

For the second part, let $a$ be a $p$-regular element, and let $X_a$
be the set of $\alpha \in X$ satisfying $\pi(\alpha) = a$. Then the
coefficient of $a$ in $(\Sigma\C)^q$ is equal to $|X_a|$. We have
already seen that $G$ acts on $X$ by conjugation; since
$\pi(\alpha^g) = \pi(\alpha)^g$, $C_G(a)$ maps $X_a$ to itself under
this action. Then $C_G(a)$ acts on $X_a$ by conjugation. Let $S$ be
a Sylow $p$-subgroup of $C_G(a)$, and let $X_a'$ be the set of fixed
points of $S$ in $X_a$. Since $S$ is a $p$-group, any orbit of $S$
in $X_a$ has size divisible by $p$ unless it consists of a single
element. Then $|X_a|$ and $|X_a'|$ are congruent modulo $p$, so
$|X_a'|$ is equal to the coefficient of $a$ in $(\Sigma\C)^q$.

Now let $\alpha \in X_a'$, and let $(x,h)$ be a defining pair for
$\alpha$; by the above, we have $(xh)_{p'} = a$. Let $s \in S$;
since $\alpha^s = \alpha$, $s$ commutes with $x^{h^{-k}}$ for every
$k \in \Z$. Conjugating by $h^k$, we get that $s^{h^k}$ commutes
with $x$ for every $k \in \Z$.

Define $b = (xh)_p$; then $a$ and $b$ commute and we have $xh =
(xh)_{p'}\cdot (xh)_p = ab$. Let $s \in S$; we prove by induction
that $s^{h^k} = s^{b^k}$ for every $k \in \N$ (and hence for every
$k \in \Z$). This is clear for $k = 0$, and for higher $k$, we get
\[
s^{h^{k+1}} = \left(s^{h^k}\right)^h = \left(s^{h^k}\right)^{xh}
= \left(s^{b^k}\right)^{ab} = \left(s^{b^k}\right)^b = s^{b^{k+1}}
\]
where we have used the fact that $x$ commutes with $s^{h^k}$, the
induction hypothesis, the equation $xh = ab$, and the fact that $a$
commutes with both $s$ and $b$.

We now see that $s^{b^k}$ commutes with $x$ for every $k \in \Z$.
Then $\<S^{b^k} \mid k \in \Z\> \subseteq C_G(x)$. Since $S$ is a
Sylow $p$-subgroup of $C_G(a)$ and $b \in C_G(a)$, Lemma
\ref{syl-konj} provides that $b \in \<S^{b^k} \mid k \in \Z\>$. Then
$b$ commutes with $x$; since it also commutes with $a$ and $b$, the
equation $xh = ab$ implies that $b$ commutes with $h$. Since both
$b$ and $h$ are $p$-elements, $hb^{-1}$ is a $p$-element, and for
any $k \in \Z$, we have
\[
x^{(hb^{-1})^{-k}} = x^{h^{-k}b^k} = x^{b^kh^{-k}} = x^{h^{-k}}
\]
Then $(x,hb^{-1})$ is a defining pair for $\alpha$, and we have
$xhb^{-1} = abb^{-1} = a$. Additionally, $hb^{-1}$ centralizes $S$,
since $s^h = s^b$ for every $s \in S$.

We conclude that any $\alpha \in X_a'$ has a defining pair $(x,h)$
such that $xh = a$ and $h$ centralizes $S$. But $x$ is determined by
$\alpha$, since it is the first element of the tuple, and $h$ is
then determined by the equation $xh = a$. Thus a defining pair of
this form is unique. Conversely, it is easily seen that any pair
$(x,h)$ such that $x \in \C$, $h$ is a $p$-element that centralizes
$S$, and $xh = a$, is a defining pair of an element of $X_a'$. Then
these pairs are in one-to-one correspondence with the elements of
$X_a'$, so in order to determine $|X_a'|$, we may simply count the
pairs instead. Since $xh = a$, these pairs are determined by $h$
alone, and an element $h \in G$ defines a valid pair if $h$ is a
$p$-element, $h$ centralizes $S$, and $x = ah^{-1} \in \C$. Since
the inverse of a $p$-element is a $p$-element, we may replace $h$ by
$h^{-1}$. This is exactly the second part of the theorem.
\end{proof}

As a corollary to the first part of this theorem, we get the
following well-known result:

\begin{cor}
Let $a \in G$ be an element that is not $p$-regular. Then $a$ has
coefficient zero in all block idempotents of $G$.
\end{cor}

\begin{thm}\label{exp-coef-zero}
Let $a$ be a $p$-regular element of $G$, let $S$ be a Sylow
$p$-subgroup of $C_G(a)$, and suppose that $C_G(S)$ contains a
normal abelian $p$-subgroup $P$ such that $a$ does not centralize
$P$. Then for any conjugacy class $\C$ of $G$ and any $p$-regular
exponent $q$ of $G$, the coefficient of $a$ in $(\Sigma\C)^q$ is
zero. In particular, $a$ has coefficient zero in all central
idempotents of $RG$.
\end{thm}

\begin{proof}
Fix a $p$-regular exponent $q$ and a conjugacy class $\C$, and let
$M$ be the set of $p$-elements $h$ of $C_G(S)$ such that $ah \in
\C$. By Theorem \ref{exp-coef}, the coefficient of $a$ in
$(\Sigma\C)^q$ is equal to $|M|$, so we wish to prove that $|M|$ is
divisible by $p$.

For each $h \in M$, we let $M_h$ be the set of elements $h'$ of $M$
such that $g^{h'} = g^h$ for all $g \in P$. The sets $M_h$ then form
a partition of $M$, so it is sufficient to show that each $M_h$ has
size divisible by $p$. Fix an $M_h$, and define a map $\varphi: P
\to P$ by $\varphi(g) = (g^{-1})^{ah}\cdot g$. Since both $a$ and
$h$ lie in $C_G(S)$ and $P$ is normal in $C_G(S)$, $\varphi(g)$ is
in fact an element of $P$, and because $P$ is abelian, $\varphi$ is
a homomorphism. Furthermore, $\varphi$ only depends on the set $M_h$
and not on the element $h$; if we choose another $h' \in M_h$, we
have $(g^{-1})^{ah'}\cdot g = (g^{-1})^{ah}\cdot g$ since
$(g^{-1})^a \in P$.

If $\varphi(P)$ is the trivial group, then $g^{ah} = g$ for each $g
\in P$; that is, the maps $g \mapsto g^a$ and $g \mapsto g^h$ are
inverses on $P$. Since $a$ is $p$-regular and does not centralize
$P$, the map $g \mapsto g^a$ is a nonidentity $p$-regular
permutation of the elements of $P$. Then its inverse should also be
a nonidentity $p$-regular permutation, but this is impossible, since
$h$ is a $p$-element. Thus $\varphi(P)$ is not the trivial group.

Since $P$ is a normal $p$-subgroup of $C_G(S)$ and $h$ is a
$p$-element, $\<h, P\>$ is a $p$-group. In particular, $h\varphi(g)$
is a $p$-element of $C_G(S)$ for every $g \in P$. Additionally, for
any $g \in P$ we have
\[
(ah)^g = g^{-1} \cdot (ah) \cdot g = (ah) \cdot (g^{-1})^{ah}g = ah\varphi(g)
\]
Since $ah \in \C$, it follows that $ah\varphi(g) \in \C$. Then
$h\varphi(g) \in M$, and since $P$ is abelian, we in fact have
$h\varphi(g) \in M_h$. We may therefore define a right action of
$\varphi(G)$ on $M_h$ by setting $h.\varphi(g) = h\varphi(g)$. This
is clearly a free action, and $\varphi(G)$ is a nontrivial
$p$-group, so it follows that $M_h$ has size divisible by $p$.

The last statement now follows from Lemma \ref{exp-block}.
\end{proof}

\section{The symmetric and alternating groups}\label{sec-symalt}

In this section, we apply Theorem \ref{exp-coef-zero} to the
symmetric and alternating groups. We begin with some well-known
results about Sylow subgroups, centralizers, and conjugation in
these groups.

\begin{lem}\label{sym-syl}
Let $P$ be a Sylow $p$-subgroup of $S_n$, and write $n =
\sum_{i=0}^k d_ip^i$ with $0 \leq d_i \leq p-1$ for each $i$. Then
$P$ is isomorphic to the direct product $\prod_{i=1}^k (W_i)^{d_i}$
where $W_i$ is isomorphic to a Sylow $p$-subgroup of $S_{p^i}$. The
action on $\{1, \ldots, n\}$ of each factor $W_i$ in $S$ has a
unique nontrivial orbit which has size $p^i$, on which $W_i$ acts
using its action as a subgroup of $S_{p^i}$, and two distinct
factors of $S$ have disjoint nontrivial orbits.
\end{lem}

\begin{lem}\label{alt-2-syl}
Let $W$ be a Sylow 2-subgroup of $S_n$. Then $W' = W \cap A_{2^k}$
is a Sylow 2-subgroup of $A_n$, $W'$ has index 2 in $W$, and if $n
\geq 4$, $W$ and $W'$ have the same orbits in $\{1, \ldots, n\}$.
\end{lem}

\begin{proof}
The first two parts are well-known. For the last part, let $i$ and
$j$ be two elements of $\{1, \ldots, n\}$ that are in the same orbit
under $W$. Pick two distinct elements $k$ and $l$ of $\{1, \ldots,
n\}$ that are both different from both $i$ and $j$; this is possible
since $n \geq 4$. Pick a $\sigma \in W$ such that $\sigma(i) = j$,
and let $\tau$ be the transposition interchanging $k$ and $l$; then
either $\sigma$ or $\tau\sigma$ is even, and $\sigma\tau(i) =
\sigma(i) = j$. Then $i$ and $j$ lie in the same orbit of $W'$.
\end{proof}

\begin{lem}\label{sym-cent}
Let $a \in S_n$ be an element that has cycle type
$c_1^{m_1}c_2^{m_2}\cdots c_r^{m_r}$. Then $C_{S_n}(a)$ is
isomorphic to $\prod_{i=1}^r C_{c_i} \wr S_{m_i}$. The embedding of
each factor into $G$ may be described thus: Fix an $l$ with $1 \leq
l \leq r$, let $M \subseteq \{1, \ldots, n\}$ be the set of points
that are moved by the $m_l$ $c_l$-cycles of $a$, and relabel the
points in $M$ by elements of $\{1, \ldots, m_l\} \times \Z/c_l$ in
such a way that $a((x,y)) = (x,y+1)$. Then $C_{c_l} \wr S_{m_l}$
acts trivially on the points of $\{1, \ldots, n\}$ outside $M$, the
$m_l$ copies of $C_{c_l}$ are generated by the elements $a_1,
\ldots, a_{m_l}$ defined by $a_x((x,y)) = (x,y+1)$ and $a_x((z,y)) =
(z,y)$ for $z \neq x$, and the subgroup $S_{m_i}$ acts on $M$ by
$\sigma((x,y)) = (\sigma(x),y)$.
\end{lem}

\begin{lem}\label{sym-konj}
Let $x$ and $y$ be elements of $S_m$, and regard $S_m$ as a subgroup
of $S_n$ for some $n \geq m$. If $x$ and $y$ are conjugate in $S_n$,
then they are conjugate in $S_m$. In particular, if $\C$ is a
conjugacy class of $S_n$, then $\C \cap S_m$ is either empty or a
conjugacy class of $S_m$.
\end{lem}

\begin{proof}
Let $x$ and $y$ be elements of $S_m$, and suppose that there is an
$s \in S_n$ such that $x^s = y$. Let $i \in \{1, \ldots, m\}$ be any
element such that $s(i) \not\in \{1, \ldots, m\}$; then $s(i)$ is a
fixed point of $x$, and we have $y(i) = x^s(i) = s^{-1}xs(i) =
s^{-1}s(i) = i$. Then $i$ is a fixed point of $y$; conversely, if $i
\in \{1, \ldots, m\}$ is not a fixed point of $y$, then $s(i) \in
\{1, \ldots, m\}$. We may therefore pick an element $t \in S_m$ such
that $t(i) = s(i)$ if $i$ is not a fixed point of $y$. Then any $i
\in \{1, \ldots, n\}$ is a fixed point of either $y$ or $s^{-1}t$,
since if it is not a fixed point of $y$, we have $s^{-1}t(i) =
s^{-1}s(i) = i$. Then $y$ and $s^{-1}t$ commute, and we get $y =
y^{s^{-1}t} = x^{ss^{-1}t} = x^t$, so $x$ and $y$ are conjugate in
$S_m$. The second part is immediate.
\end{proof}

\begin{lem}\label{alt-konj}
Let $x$ and $y$ be elements of $A_m$, and regard $A_m$ as a subgroup
of $A_n$ for some $n \geq m$. If $x$ and $y$ are conjugate in $A_n$,
then either they are conjugate in $A_m$, or there exists an odd
element $g \in S_m$ such that $x = y^g$. In particular, if $\C$ is a
conjugacy class of $A_n$, then $\C \cap A_m$ is either empty, a
conjugacy class of $A_m$, or a union of two conjugacy classes $\C_1$
and $\C_2$ of $A_m$ such that $\C_1^g = \C_2$ for every odd element
$g \in S_m$.
\end{lem}

\begin{proof}
Suppose $x$ and $y$ are elements of $A_m$ that are conjugate in
$A_n$; then by Lemma \ref{sym-konj}, they are conjugate in $S_m$.
This is the first part. For the second part, we note that $S_m$ acts
transitively on $\C \cap A_m$, and since $A_m$ is normal in $S_m$,
$S_m$ acts on the set of $A_m$-orbits in $\C \cap A_m$ (that is, the
set of $A_m$-conjugacy classes contained in $\C \cap A_m$). $A_m$
acts trivially on this set, while $S_m$ acts transitively, so
$S_m/A_m$ acts transitively on this set. Since $S_m/A_m$ has order
2, there are at most two $A_m$-conjugacy classes in $\C \cap A_m$,
and if there are two, they are interchanged by the nontrivial
element of $S_m/A_m$.
\end{proof}

\begin{thm}\label{symalt-idem}
Let $G = S_n$ or $G = A_n$, and let $a$ be a $p$-regular element of
$G$ containing at least $p$ $c$-cycles for some $c > 1$. In the case
$G = A_n$ and $p = 2$, assume further that $a$ contains at least 4
$c$-cycles, or there exists a number $d \neq c$ such that $a$
contains at least 2 $d$-cycles. Then $a$ does not appear in any
central idempotent of $RG$.
\end{thm}

\begin{proof}
This will be proven by showing that $a$ satisfies the condition in
Theorem \ref{exp-coef-zero}. We first consider the case $G = S_n$.

Let $a$ have cycle type $c_1^{m_1}c_2^{m_2}\cdots c_r^{m_r}$, and
write $C_{S_n}(a)$ as $\prod_{i=1}^r C_{c_i} \wr S_{m_i}$ as in
Lemma \ref{sym-cent}. Choose a $k$ such that $c_k > 1$ and $m_k \geq
p$; this is possible by assumption. Since $a$ is $p$-regular, $c_k$
is not divisible by $p$, so the subgroup $S_{m_k}$ of $C_{c_k} \wr
S_{m_k}$ contains a Sylow $p$-subgroup $T$ of $C_{c_k} \wr S_{m_k}$.
Let $S$ be a Sylow $p$-subgroup of $C_{S_n}(a)$ containing $T$;
since $C_{c_k} \wr S_{m_k}$ is a direct factor of $C_{S_n}(a)$, $T$
is then a direct factor of $S$. By the above, there is a complement
of $C_{c_k} \wr S_{m_k}$ in $C_{S_n}(a)$ that acts trivially on $M$,
so there is a complement of $T$ in $S$ that acts trivially on $M$.

Let $W$ be a direct factor of $T$ in the decomposition given in
Lemma \ref{sym-syl}; such a factor exists because $m_k \geq p$.
Under the usual action of $S_{m_k}$ on $\{1, \ldots, m_k\}$, there
is then a subset $N$ of $\{1, \ldots, m_k\}$ with $|N|$ a nonzero
power of $p$ such that $W$ acts transitively on $N$ and fixes all
other points. Additionally, there is a complement of $W$ in $T$ that
acts trivially on $N$. Transferring this to the action on $\{1,
\ldots, n\}$ by using the relabeling given in Lemma \ref{sym-cent},
we see that $W$ acts on $N \times \Z/c_k$ with orbits $N \times
\{1\}, N \times \{2\}, \ldots, N \times \{c_k\}$, and it acts
trivially on all other points of $\{1, \ldots, n\}$. Combining this
with the previous paragraph, we see that $W$ is a direct factor of
$S$, and there is a complement of $W$ in $S$ that acts trivially on
$N \times \Z/c_k$.

The group $W$ acts on $N$, and hence, so does $Z(W)$. We use this
action to define $c_k$ different actions (labelled by the elements
of $\Z/c_k$) of $Z(W)$ on $\{1, \ldots, n\}$. Still maintaining the
above relabeling, we first specify that $Z(W)$ in all cases acts
trivially on points outside $N \times \Z/c_k$. On $N \times \Z/c_k$,
the $i$'th action is defined by $\sigma((x,i)) = (\sigma(x),i)$ and
$\sigma((x,j)) = (x,j)$ for $j \neq i$. We now define the $c_k$
subgroups $Z_1, \ldots, Z_{c_k}$ of $G = S_n$ by letting $Z_i$ be
the image of the natural map from $Z(W)$ into $S_n$ constructed from
the $i$'th action. Then $Z_i$ acts transitively on $N \times \{i\}$
and trivially on all other points of $\{1, \ldots, n\}$. Since an
element of $Z(W)$ is completely defined by its action on $N$, each
$Z_i$ is isomorphic to $Z(W)$.

Since $N \times \{i\}$ and $N \times \{j\}$ are disjoint for $i \neq
j$, $Z_i$ and $Z_j$ centralize each other and have trivial
intersection. Define $P = \<Z_i \mid 1 \leq i \leq c_k\>$; $P$ is
then isomorphic to the direct product $\prod_{i=1}^{c_k} Z_i$.
Furthermore, each $Z_i$ is an nontrivial abelian $p$-group, since it
is isomorphic to $Z(W)$. Then $P$ is an nontrivial abelian
$p$-group.

It is straightforward to verify that $W$, considered as a subgroup
of $S_n$, centralizes the groups $Z_1, \ldots, Z_{c_k}$.
Additionally, $W$ has a complement in $S$ that acts trivially on $N
\times \Z/c_k$, and since the $Z_i$ act trivially outside $N \times
\Z/c_k$, this complement also centralizes the $Z_i$. Then $S$
centralizes the $Z_i$, so we have $Z_i \subseteq C_{S_n}(S)$. This
implies $P \subseteq C_{S_n}(S)$.

Now let $g \in C_{S_n}(S)$. Then $g$ centralizes $W$, so by Lemma
\ref{cent-orb}, $g$ permutes the sets $N \times \{1\}, N \times
\{2\}, \ldots, N \times \{c_k\}$. Let $i \in Z/c_k$ be arbitrary,
and choose $j \in Z/c_k$ such that $g^{-1}(N \times \{i\}) = N
\times \{j\}$; then $Z_i^g$ acts trivially on all points outside $N
\times \{j\}$. Let $z \in Z_i$; because of the way we have defined
$Z_i$, there exists an element $\tilde z \in Z(W)$ (now considered
as a subgroup of $S_n$) such that $z((x,i)) = \tilde z((x,i))$ for
all $x \in N$. Then we also have $z^g((x,j)) = \tilde z^g((x,j))$
for all $x \in N$, and since $g$ centralizes $W$, we have $\tilde
z^g = \tilde z$. Then $z^g((x,j)) = \tilde z((x,j))$, and $z^g$
fixes all points outside $N \times \{j\}$. Thus $z^g$ is an element
of $Z_j$, so $Z_i^g = Z_j$. Then $g$ permutes the groups $Z_1,
\ldots, Z_{c_k}$, so $g$ normalizes $P$. This implies that $P$ is
normal in $C_{S_n}(S)$.

Finally, we note that since the action of $a$ on $N \times \Z/c_k$
is given by $a((x,y)) = (x,y+1)$, we have $Z_i^a = Z_{i-1}$. As $c_k
> 1$, this implies that $a$ does not centralize $P$. Then $P$
satisfies all the conditions of Theorem \ref{exp-coef-zero}.

This completes the proof for $G = S_n$. In the case $p > 2$, we see
that every $p$-element of $S_n$ is contained in $A_n$; in
particular, the groups $S$ and $P$ constructed above are contained
in $A_n$. Further, $S$ is a Sylow $p$-subgroup of $C_{A_n}(a)$
because it is a Sylow $p$-subgroup of $C_{S_n}(a)$, and $P$ is
normal in $C_{A_n}(S)$ because it is normal in $C_{S_n}(S)$. We may
then apply Theorem \ref{exp-coef-zero} to prove the case $G = A_n$
and $p > 2$.

For the case $G = A_n$ and $p = 2$, we will need to modify the
proof. Let again $a$ have cycle type $c_1^{m_1}c_2^{m_2}\cdots
c_r^{m_r}$, and write $C_{S_n}(a) \cong \prod_{i=1}^r C_{c_i} \wr
S_{m_i}$. For each $i$, let $T_i$ be a Sylow 2-subgroup of
$S_{m_i}$. Since $a$ is 2-regular, each $c_i$ is odd, so $T_i$ is
also a Sylow 2-subgroup of $C_{c_i} \wr S_{m_i}$, and $S =
\prod_{i=1}^r T_i$ is a Sylow 2-subgroup of $C_{S_n}(a)$. The group
$S' = S \cap A_n$ is then a Sylow 2-subgroup of $C_G(a)$.

Suppose that we can show that $C_{A_n}(S') = C_{A_n}(S)$. We can
then take a $k$ with $m_k \geq 2$ and use it to construct the group
$P \cong \prod_{i=1}^{c_k} Z_i$ as above, although it will only be a
subgroup of $S_n$, not necessarily of $A_n$. If we set $P' = P \cap
A_n$, we have $P' \lhd C_{A_n}(S)$, since $P \lhd C_{S_n}(S)$. As
$C_{A_n}(S') = C_{A_n}(S)$, we get $P' \lhd C_{A_n}(S')$.
Additionally, $Z_i \times Z_{i+1}$ is a subgroup of $P$ of order at
least 4, and since $P'$ has index at most 2 in $P$, $P' \cap (Z_i
\times Z_{i+1})$ is nontrivial. As $(P' \cap (Z_i \times Z_{i+1}))^a
= P' \cap (Z_{i-1} \times Z_i)$ and $c_k \geq 3$, we find that $a$
does not centralize $P'$. We can then apply Theorem
\ref{exp-coef-zero}, completing the proof. This method can be used
to cover most subcases, but a few will need to be treated
separately.

We will need a small observation regarding the parity of elements in
$T_i$: if $x \in T_i$, then $x$ is an odd element of $S_n$ if and
only if it is an odd element of $S_{m_i}$. To see this, note that
the action of $x$ on $\{1, \ldots, n\}$ is given by identifying
$\{1, \ldots, m_i\}$ with $c_i$ pairwise disjoint subsets of $\{1,
\ldots, n\}$ and letting $x$ act on each of these using its action
on $\{1, \ldots, m_i\}$. Then $x$ in $S_n$ is the product of $c_i$
elements, each of which has the same parity as $x$ in $S_{m_i}$. As
$c_i$ is odd, the result follows.

Assume first that $a$ contains 4 $c$-cycles for some $c > 1$, so
there is a $k$ such that $m_k \geq 4$ and $c_k > 1$. By Lemma
\ref{sym-syl}, we may write $T_k$ as a direct product $\prod_{j=1}^t
W_j$ where each $W_j$ is a Sylow 2-subgroup of $S_{2^{n_j}}$ for
some $n_j$, and no two $n_j$ are equal. Assume without loss of
generality that $W_1$ has the largest order among the $W_j$; in
particular, the order of $W_1$ is at least 4, and the nontrivial
orbits of the action of $W_1$ on $\{1, \ldots, n\}$ have size at
least 4. We set $W = W_1$ and $W' = W \cap A_n$; by Lemma
\ref{alt-2-syl} and the above observation regarding parity of
elements, $W'$ is then a proper subgroup of $W$. Since $W$ has order
at least 4, we further have that $W'$ has the same orbits as $W$ in
$\{1, \ldots, n\}$.

Consider the case that there is an $l$ distinct from $k$ such that
$m_l \geq 2$, so we may pick an element $u \in T_l$ that is a
transposition when considered as an element of $S_{m_l}$. Then $u$
is odd when considered as an element of $S_n$, so $S'$ contains the
elements $wu$ where $w \in W$ but $w \not\in W'$. Since $u$ commutes
with every element of $W$, we may define the subgroup $W_u$ of $S'$
to consist of these elements together with $W'$. Let $x \in
C_{A_n}(S')$; then $x$ centralizes $W_u$, so it must permute the
orbits of $W_u$ in $\{1, \ldots, n\}$. The orbits of $W_u$ consist
of a number of orbits of size 2 from the action of $u$, and a number
of orbits from the action of $W$, plus some fixed points. Since the
orbits from $W$ have size at least 4, they are permuted separately
from the orbits from $u$. Then $x$ permutes the nontrivial orbits of
the action of $u$ alone, so it follows that $x$ actually commutes
with $u$. As $x$ then centralizes both $S'$ and $u$, it must
centralize $S$. Then we have $C_{A_n}(S') = C_{A_n}(S)$.

Now consider the case that $m_i < 2$ for $i \neq k$ but $t > 1$.
Then $S'$ contains those elements $(w_1,w_2) \in W \times W_2$ for
which $w_1$ and $w_2$ are both even or both odd. Similarly to the
previous case, we note that the orbits of $S' \cap (W \times W_2)$
consist of a number of orbits from $W$ and a number of orbits from
$W_2$, and the orbits from $W$ do not have the same size as the
orbits from $W_2$. As before, we then conclude that an element that
centralizes this group in fact centralizes $W \times W_2$, from
which it follows that $C_{A_n}(S') = C_{A_n}(S)$.

There remains the case that $m_i < 2$ for $i \neq k$ and $t = 1$.
Here $S = W$ and $S' = W'$, and $W'$ is nontrivial because $W$ has
order at least 4. Then $Z(W')$ is nontrivial, and we use this group
to construct a normal abelian $p$-subgroup of $C_{A_n}(S')$
satisfying the conditions of Theorem \ref{exp-coef-zero}, in the
same manner as in the proof of the case $G = S_n$.

Now consider the possibility that $m_i < 4$ for all $i$, but there
exists two distinct numbers $k$ and $l$ such that $m_k \geq 2$ and
$m_l \geq 2$. In this case, each nontrivial $T_i$ is cyclic of order
2, and the nontrivial element is a product of $c_i$ disjoint
transpositions. Suppose first that there exists a third number $m$,
distinct from both $k$ and $l$, such that $T_m$ has order 2. Let
$s_k$, $s_l$, and $s_m$ be the nontrivial elements of $T_k$, $T_l$,
and $T_m$, respectively. These are all odd, being the product of an
odd number of transpositions, so $s_ks_l$ and $s_ks_m$ are even and
therefore elements of $S'$. Now let $x \in C_{A_n}(S')$. The orbits
of $s_ks_l$ consist of $c_k$ orbits of size 2 from $s_k$ and $c_l$
orbits of size 2 from $s_l$, plus fixed points, and $x$ must permute
the orbits of size 2. The orbits from $s_k$ are also nontrivial
orbits of $s_ks_m$, while the orbits from $s_l$ are fixed points of
$s_ks_m$. As $x$ also centralizes $s_ks_m$, $x$ then cannot map an
orbit from $s_k$ into an orbit from $s_l$. Then $x$ permutes the
orbits from $s_k$, which implies that $x$ centralizes $s_k$. Then we
have $C_G(S') = C_G(S)$.

Finally we have the possibility that there is no such third number.
Then $S$ is isomorphic to $T_k \times T_l \cong C_2 \times C_2$, and
$S'$ is cyclic of order 2, with the nontrivial element being a
product of $c_k + c_l$ disjoint transposistions. Then we may write
$C_{S_n}(S') \cong S_{n-c_k-c_l} \times (C_2 \wr S_{c_k+c_l})$, and
$C_{A_n}(S') = C_{S_n}(S') \cap A_n$. Clearly $C_{S_n}(S')$ has the
normal subgroup $C_2^{c_k+c_l}$, and setting $P = C_2^{c_k+c_l} \cap
A_n$, we obtain a normal abelian 2-subgroup of $C_{A_n}(S')$. As
either $c_k \geq 3$ or $c_l \geq 3$, we see as in the first case
above that $a$ does not centralize $P$.
\end{proof}

\begin{lem}\label{symalt-idem-syl}
Let $G = S_n$ or $G = A_n$, and let $e$ be a block idempotent of
$G$. Then there exists a subset $M$ of $\{1, \ldots, n\}$ such that
a Sylow $p$-subgroup of $S_M$ (respectively $A_M$) is a defect group
of $e$.
\end{lem}

\begin{proof}
By Lemma \ref{def-syl}, we may choose an element $a$ with nonzero
coefficient in $e$ such that the Sylow $p$-subgroups of $C_G(a)$ are
defect groups of $e$; write the cycle type of $a$ as
$1^mc_1^{m_1}\cdots c_r^{m_r}$. By Theorem \ref{symalt-idem}, we
then have $m_i < p$ for all $i$ (except possibly in the case $G =
A_n$ and $p = 2$). Write $C_{S_n}(a) \cong S_m \times \prod_{i=1}^r
C_{c_i} \wr S_{m_i}$; then the factor $\prod_{i=1}^r C_{c_i} \wr
S_{m_i}$ has order not divisible by $p$, so the Sylow $p$-subgroups
of $C_{S_n}(a)$ are simply the Sylow $p$-subgroups of $S_m$. It
follows that a Sylow $p$-subgroup of $S_m$ is a defect group of $e$.
Since $S_m$ embeds into $S_n$ as the subgroup consisting of
permutations of the fixed points of $a$, we choose $M$ to be the set
of fixed points of $a$. This obviously works for $G = S_n$, and it
also works for $G = A_n$ because the Sylow $p$-subgroups of
$C_{A_n}(a)$ are the Sylow $p$-subgroups of $A_n \cap S_M = A_M$.

In the exceptional case $G = A_n$ and $p = 2$, it is possible that
there is a $k$ such that $m_k \geq 2$. If this happens, we must have
$m_i < 2$ for all $i \neq k$, $m_k < 4$, and $m < 2$. Then
$C_{S_n}(a)$ has order divisible by 2 but not by 4, and the
nontrivial element of a Sylow 2-subgroup is a product of $c_k$
transpositions, which is odd because $c_k$ is odd. This implies that
$C_{A_n}(a) = A_n \cap C_{S_n}(a)$ contains no nontrivial elements
of order 2, so its Sylow 2-subgroup is trivial. Then we may choose
$M$ to be empty.
\end{proof}

This is the restriction we need on the possible defect groups of
blocks of symmetric groups. To determine the possible fusion
systems, we will use Alperin's fusion theorem. First, we will
determine what the centric subgroups of the defect group look like.

\begin{lem}\label{symalt-centric}
Let $P$ be a Sylow $p$-subgroup of $S_{pm}$ or $A_{pm}$ and let $Q$
be a centric subgroup of $P$. Then $Q$ has no fixed points in $\{1,
\ldots, pm\}$ and $C_{S_{pm}}(Q)$ is a $p$-group.
\end{lem}

\begin{proof}
We first consider the Sylow $p$-subgroups of $S_{pm}$. Since Sylow
$p$-subgroups are conjugate, it is sufficient to prove the claim for
a particular Sylow $p$-subgroup.

Let $z$ be an element of $S_{pm}$ consisting of $m$ $p$-cycles, and
let $C = C_{S_{pm}}(z)$. By Lemma \ref{sym-cent}, $C$ is then
isomorphic to $C_p \wr S_m \cong (C_p)^m \rtimes S_m$. We let $a_1,
\ldots, a_m$ be the generators of the $m$ copies of $C_p$, so that
$z = \prod_{i=1}^m a_i$, and we have the natural homomorphism
$\varphi: (C_p)^m \rtimes S_m \to S_m$ whose kernel is $(C_p)^m$.

The order of $C$ is $p^mm! = \prod_{i=1}^m pi$, which is equal to
the product of those factors in $(pm)!$ that are divisible by $p$.
Then $|C|$ is divisible by the same power of $p$ as $|S_{pm}|$, so a
Sylow $p$-subgroup of $C$ is also a Sylow $p$-subgroup of $S_{pm}$.
We pick one such group $P = (C_p)^m \rtimes T$, where $T$ is a Sylow
$p$-subgroup of $S_m$. Clearly, $z$ is an element of $Z(P)$.

Let $Q$ be a centric subgroup of $P$. Then $Q$ contains $Z(P)$, so
$z$ is an element of $Q$. Since $z$ has no fixed points, neither
does $Q$.

Furthermore, $C_{S_{pm}}(Q)$ is contained in $C$. Suppose that
$C_{S_{pm}}(Q)$ is not a $p$-group; we let $x$ be a nontrivial
element of $C_{S_{pm}}(Q)$ of order not divisible by $p$. Since $x
\in C$, we may write $x = bs$ with $b \in (C_p)^m$ and $s \in S_m$.
Then $s$ is a nontrivial element of $S_m$ with order not divisible
by $p$, since it is the image of $a$ under $\varphi$. Pick a $k \in
\{1, \ldots, m\}$ that is not a fixed point of $s$, and let $K$ be
the orbit of $k$ under $\varphi(Q)$.

If $s(k)$ is an element of $K$, there is a $t \in \varphi(Q)$ such
that $t(k) = s(k)$. But since $a$ centralizes $Q$, $s = \varphi(x)$
must centralize $\varphi(Q)$, so $s$ and $t$ commute. By induction
we then get $t^n(k) = t(t^{n-1}(k)) = t(s^{n-1}(k)) = s^{n-1}(t(k))
= s^{n-1}(s(k)) = s^n(k)$ for all $n \in \N$. That is, the cycles in
$s$ and $t$ containing $k$ are identical. But this is impossible
since $s$ has order not divisible by $p$ and does not fix $k$, and
$t$ is an element of a $p$-group. So $s(k) \not\in K$.

Now let $a = \prod_{i \in K} a_i$. By the definition of $K$, $a$
commutes with every element of $\varphi(Q)$; it also commutes with
every element of $(C_p)^m$, since this group is abelian. Since $Q$
is contained in $(C_p)^m \rtimes \varphi(Q)$, $a$ then centralizes
$Q$. As $a \in P$ and $Q$ is centric in $P$, it follows that $a \in
Q$. This implies that $a$ commutes with $x = bs$; since it also
commutes with $b$, it commutes with $s$. But this implies $s(K) =
K$, which is a contradiction since $s(k) \not\in K$.

This covers $S_{pm}$. When $p > 2$, $S_{pm}$ and $A_{pm}$ have the
same Sylow $p$-subgroups, so the only case that remains is the Sylow
2-subgroups of $A_{2m}$.

Consider first the case that $m$ is even, so we may write $2m = 4l$
for some $l$. Let $z$ be the product of $2l$ disjoint
transpositions; this is an even element, so it lies in $A_{4l}$. Set
$C' = C_{S_{4l}}(z)$; then $C'$ is isomorphic to $C_2 \wr S_{2l}
\cong (C_2)^{2l} \rtimes S_{2l}$, and if $T$ is a Sylow 2-subgroup
of $S_{2l}$, then $P' = (C_2)^{2l} \rtimes T$ is a Sylow 2-subgroup
of both $C'$ and $S_{4l}$. Set $C = C' \cap A_{4l}$ and $P = P' \cap
A_{4l}$; then $C = C_{A_{4l}}(z)$ and $P$ is a Sylow 2-subgroup of
$A_{4l}$ such that $z \in Z(P)$. We have the natural homomorphism
$\varphi: (C_2)^{2l} \rtimes S_{2l} \to S_{2l}$.

Let $Q$ be a centric subgroup of $P$; then $Q$ contains $z$, so
$C_{S_{4l}}(Q)$ is contained in $C'$. Suppose that $C_{S_{4l}}(Q)$
is not a 2-group, and let $x$ be a nontrivial element of odd order
of $C_{S_{4l}}(Q)$. As in the case of $S_{pm}$, we write $x = bs$
with $b \in (C_2)^{2l}$ and $s \in S_{2l}$, and we pick a $k \in
\{1, \ldots, 2l\}$ that is not a fixed point of $s$. Assume for the
moment that we can pick $k$ such that it is not a fixed point of
$\varphi(Q)$. We then let $K$ be the orbit of $k$ under
$\varphi(Q)$, and let $a = \prod_{i \in K} a_i$ where $a_1, \ldots,
a_{2l}$ are the generators of $(C_2)^{2l}$. The size of $K$ is a
power of 2 different from 1, so it is even. Then $a$ is an element
of $A_{2l}$, and by the same arguments as in the case of $S_{pm}$,
we conclude both that $s(K) = K$ and $s(k) \not\in K$, which is a
contradiction.

There remains the possibility that every element in $\{1, \ldots,
2l\}$ that is not a fixed point of $s$ is a fixed point of
$\varphi(Q)$. We then let $K$ be the set of points that are not
fixed points of $s$, and we write $(C_2)^K$ for the subgroup of
$(C_2)^{2l}$ generated by the elements $a_i$ as $i$ runs over $K$.
Then $R = (C_2)^K \cap A_{2l}$ is a nontrivial subgroup of $P$; by
the definition of $K$, any element of $\varphi(Q)$ centralizes this
group. Then every element of $R$ centralizes both $(C_2)^{4l}$ and
$\varphi(Q)$, so they all centralize $Q$. As $Q$ is centric in $P$,
we get $R \subseteq Q$, so $x$ centralizes $R$. Since $b \in
(C_2)^{2l}$, $b$ also centralizes $R$, and hence, so does $s$. But
if we take any $k \in K$, $k$ is a part of a cycle in $s$ of length
at least 3, so $a_ka_{s(k)}$ is an element of $R$ that does not
commute with $s$. This is a contradiction.

We now consider the case that $m$ is odd, so we may write $2m = 4l +
2$. Let $z$ be a product of $2l$ transpositions, and set $C' =
C_{S_{4l+2}}(z)$; then $C'$ is isomorphic to $(C_2 \wr S_{2l})
\times C_2$. We let $w$ be the nontrivial element of the $C_2$
factor, so that $w$ is the transposition interchanging the two fixed
points of $z$. Let $P'$ be a Sylow 2-subgroup of $C'$; comparing
orders, we find that $P'$ is also a Sylow 2-subgroup of $S_{4l+2}$.
Set $C = C' \cap A_{4l+2}$ and $P = P' \cap A_{4l+2}$; then $P$ is a
Sylow 2-subgroup of both $C$ and $A_{4l+2}$. Since $z$ is a central
element of $C$, any centric subgroup of $P$ contains $z$, so
$C_{A_{4l+2}}(Q)$ is contained in $C$. Since $C_{A_{4l+2}}(Q)$ has
index at most 2 in $C_{S_{4l+2}}(Q)$, it is sufficient to prove that
$C_{A_{4l+2}}(Q)$ is a 2-group.

Now consider the projection homomorphism $\pi: C' \to C_2 \wr
S_{2l}$ with kernel $C_2 = \<w\>$. For any element $x \in C'$,
exactly one of $x$ and $xw$ is even, since $w$ is odd. This implies
that when we restrict $\pi$ to $C$, we obtain an isomorphism $C
\cong C_2 \wr S_{2l}$. Then $\pi(P)$ is a Sylow 2-subgroup of $C_2
\wr S_{2l}$, and it is enough to prove that the centralizer in $C_2
\wr S_{2l}$ of any centric subgroup of $\pi(P)$ is a 2-group. But
this was shown as part of the case $G = S_{2l}$ and $p = 2$.
\end{proof}

We can now determine the number of Brauer pairs at a centric
subgroup of the defect group. This will allow us to determine the
automorphism groups of these pairs, which is sufficient to determine
the block fusion system.

\begin{lem}\label{sym-br}
Let $e$ be a block idempotent of $S_n$, let $P$ be a defect group of
$e$, and let $Q$ be a centric subgroup of $P$. Then there exists a
unique Brauer pair at $Q$ associated to $e$.
\end{lem}

\begin{proof}
Let $N \subseteq \{1, \ldots, n\}$ be the set of fixed points of $P$
and let $M$ be its complement. Then $P$ is a Sylow $p$-subgroup of
$S_M$ and $|M|$ is divisible by $p$, so by Lemma
\ref{symalt-centric}, $C_{S_n}(Q)$ is isomorphic to $T \times S_N$
where $T$ is a $p$-group. Then every $p$-regular element of
$C_{S_n}(Q)$ is contained in $S_N$, so we may identify the central
idempotents of $kC_{S_n}(Q)$ with the central idempotents of $kS_N$.

Since the map $\Br_Q: (kS_n)^Q \to kC_{S_n}(Q)$ is surjective, it
restricts to a map $\Br_Q: Z(kS_n) \to Z(kC_{S_n}(Q))$. By Lemma
\ref{sym-konj}, any element of $Z(kS_N)$ lies in the image of this
restricted map. This includes the primitive idempotents of
$Z(kC_{S_n}(Q))$, so these can all be lifted to a primitive
idempotent of $kS_n$. Conversely, $\Br_Q$ then maps any primitive
idempotent of $Z(kS_n)$ to either zero or a primitive idempotent. In
particular, $\Br_Q(e)$ is a primitive idempotent of
$Z(kC_{S_n}(Q))$, so there is a unique Brauer pair at $Q$ associated
to $e$.
\end{proof}

\begin{thm}\label{sym-fusion}
Let $e$ be a block idempotent of $S_n$, let $P$ be a defect group of
$e$, and let $M$ be a subset of $\{1, \ldots, n\}$ such that $P$ is
a Sylow $p$-subgroup of $S_M$. Then the block fusion system on $P$
is equal to $\F_P(S_M)$.
\end{thm}

\begin{proof}
Let $\F$ be the block fusion system on $P$. By Alperin's theorem, it
is sufficient to prove that $\Aut_\F(Q) = \Aut_{S_M}(Q)$ for each
centric subgroup $Q$ of $P$. For each of these groups, there is a
unique Brauer pair $(Q, e_Q)$ associated to $e$, so all elements of
$S_n$ that normalize $Q$ also normalize $(Q, e_Q)$. Then we have
$\Aut_\F(Q) = \Aut_{S_n}(Q)$. Since $S_M \subseteq S_n$, we have
$\Aut_{S_M}(Q) \subseteq \Aut_{S_n}(Q)$, so we need only prove the
reverse inclusion. Let $\varphi \in \Aut_{S_n}(Q)$ be an
automorphism represented by the element $g \in S_n$. Let $N$ be the
set of fixed points of $Q$; then $M \cup N = \{1, \ldots, n\}$,
although the two sets need not be disjoint. Since $g$ normalizes
$Q$, $g$ maps $N$ to itself, so we may consider the element $h \in
S_n$ defined by $h(x) = g(x)$ for $x \in N$ and $h(x) = x$
otherwise. Since $h$ only permutes the fixed points of $Q$, it must
centralize $Q$. That is, $h$ represents the identity automorphism of
$Q$. Then $gh^{-1}$ represents $\varphi$, and we have $gh^{-1}(x) =
x$ for $x \in N$. This implies that $gh^{-1}$ is an element of
$S_M$, so it follows that $\Aut_{S_n}(Q) \subseteq \Aut_{S_M}(Q)$.
\end{proof}

\begin{lem}\label{alt-br}
Let $e$ be a block idempotent of $A_n$, let $P$ be a defect group of
$e$, and let $Q$ be a centric subgroup of $P$. Then there exist
either one or two Brauer pairs af $Q$ associated to $e$, this number
being the same for all $Q$. If there are two, then conjugation by $g
\in N_{A_n}(Q)$ interchanges the two pairs if and only if the action
of $g$ on the fixed points of $P$ is an odd permutation.
\end{lem}

\begin{proof}
Let $N$ be the set of fixed points of $P$. By Lemma
\ref{symalt-centric}, $C_{S_n}(Q)$ is the direct product of a
$p$-group with $S_N$, so every $p$-regular element of $C_{S_n}(Q)$
is contained in $S_N$. Then all $p$-regular elements of $C_{A_n}(Q)$
are contained in $A_n \cap S_N = A_N$, so we may identify the
central idempotents of $kC_{A_n}(Q)$ with the central idempotents of
$kA_N$. Since $N$ does not depend on the choice of $Q$, this shows
that there is the same number of Brauer pairs at $Q$ associated to
$e$ for all possible $Q$. We also note that since the map $\Br_Q:
(kA_n)^Q \to kC_{A_n}(Q)$ is surjective, it restricts to a map
$\Br_Q: Z(kA_n) \to Z(kC_{A_n}(Q))$.

Since $A_N$ is normal in $S_N$, $S_N$ acts on $kA_N$ by conjugation,
and $A_N$ obviously acts trivially on $Z(kA_N)$ under this action.
Hence $S_N/A_N$ acts on $Z(kA_N)$; let $\sigma$ be the nontrivial
element of $S_N/A_N$. Let $x$ be any element of $Z(kA_N)$. Then $x$
is also an element of $Z(kC_{A_n}(Q))$, and by Lemma \ref{alt-konj},
$x$ lies in $\Br_Q(Z(kA_n))$ if and only if $x^\sigma = x$.

Now let $(Q,e_Q)$ be a Brauer pair at $Q$ associated to $e$; then
$e_Q^\sigma$ is again a central idempotent of $kA_N$. Suppose first
that $e_Q^\sigma = e_Q$. Then $e_Q$ lies in $\Br_Q(Z(kA_n))$, and
therefore lifts to a primitive central idempotent of $kA_n$. This
primitive central idempotent must be $e$, since $(Q,e_Q)$ is
associated to $e$. That is, we have $\Br_Q(e) = e_Q$, so there is a
unique Brauer pair at $Q$ associated to $e$.

Now suppose that $e_Q^\sigma \neq e_Q$. Since $\sigma^2$ is trivial,
$\sigma$ then interchanges $e_Q$ and $e_Q^\sigma$. As $e_Q$ and
$e_Q^\sigma$ are distinct primitive central idempotents of $kA_N$,
$e_Q + e_Q^\sigma$ is a central idempotent of $kA_N$. Since $(e_Q +
e_Q^\sigma)^\sigma = e_Q^\sigma + e_Q$, $e_Q + e_Q^\sigma$ lies in
$\Br_Q(Z(kA_n))$. Additionally, it must be a primitive idempotent in
$\Br_Q(Z(kA_n))$, since $e_Q$ does not lie in $\Br_Q(Z(kA_n))$ and
both $e_Q$ and $e_Q^\sigma$ are primitive in $Z(kC_{A_n}(Q))$. Then
$e_Q + e_Q^\sigma$ lifts to a primitive idempotent of $Z(kA_n)$;
since $(Q,e_Q)$ is associated to $e$, this idempotent must be $e$.
That is, we have $\Br_Q(e) = e_Q + e_Q^\sigma$, so there are two
Brauer pairs at $Q$ associated to $e$, namely $(Q,e_Q)$ and
$(Q,e_Q^\sigma)$.

For any $h \in N_{A_n}(Q)$, conjugation by $h$ must either fix both
of these pairs or interchange them, depending on whether $e_Q^h$ is
equal to $e_Q$ or $e_Q^\sigma$. This can only depend on the action
of $h$ on $N$, since $e_Q$ lies in $kA_N$, and by the above, we have
$e_Q^h = e_Q^\sigma$ if and only if $h$ restricts to an odd element
of $S_N$.
\end{proof}

\begin{thm}\label{alt-fusion}
Let $e$ be a block idempotent of $A_n$, let $P$ be a defect group of
$e$, and let $M$ be the subset of $\{1, \ldots, n\}$ consisting of
those elements that are not fixed points of $P$. Then $P$ is a Sylow
$p$-subgroup of $A_M$, and the block fusion system on $P$ is equal
to either $\F_P(A_M)$ or $\F_P(S_M)$. If $p = 2$, then the block
fusion system on $P$ is equal to $\F_P(A_M)$.
\end{thm}

\begin{proof}
By Lemma \ref{symalt-idem-syl}, there is a subset $M'$ of $\{1,
\ldots, n\}$ such that $P$ is a Sylow $p$-subgroup of $A_{M'}$. Then
we clearly have $M \subseteq M'$, so $A_M \subseteq A_{M'}$, and
since $P \subseteq A_M$, $P$ is a Sylow $p$-subgroup of $A_M$.

Let $\F$ be the block fusion system on $P$. By Alperin's theorem, it
is enough to prove that either $\Aut_\F(Q) = \Aut_{S_M}(Q)$ for
every centric subgroup $Q$ of $P$, or $\Aut_\F(Q) = \Aut_{A_M}(Q)$
for all these groups.

Let $N$ be the set of fixed points of $P$, so that $\{1, \ldots,
n\}$ is the disjoint union of $M$ and $N$. By Lemma
\ref{symalt-centric}, any centric subgroup of $Q$ has $N$ as its set
of fixed points. Then any element $x \in N_{A_n}(Q)$ satisfies $x(N)
= N$, so we may write $x = ab$ with $a \in S_M$ and $b \in S_N$ with
$a$ and $b$ either both even or both odd. Since $b$ centralizes $Q$,
$x$ and $a$ represent the same automorphism of $Q$, so we always
have $\Aut_\F(Q) \subseteq \Aut_{A_n}(Q) \subseteq \Aut_{S_M}(Q)$.

Suppose that at every $Q$ there is a unique Brauer pair at $Q$
associated to $e$. Then we have $\Aut_\F(Q) = \Aut_{A_n}(Q)$. If $N$
contains at least two elements, there is an element $b \in S_N$ of
odd order. Then for any $a \in S_M$, either $a$ or $ab$ is an
element of $G$, and $ab$ represents the same automorphism of $Q$ as
$a$. We then have $\Aut_{S_M}(Q) \subseteq \Aut_{A_n}(Q)$; combining
this with the earlier results, we get $\Aut_\F(Q) = \Aut_{S_M}(Q)$.
If $N$ consists of at most one element, $S_N$ is trivial, and we get
$N_{A_n}(Q) = N_{A_M}(Q)$. This implies $\Aut_\F(Q) = \Aut_{A_M}(Q)$
for all $Q$.

Now suppose that at every $Q$ there are two Brauer pairs $(Q,e_Q)$
and $(Q,e_Q')$ associated to $e$. Let $x \in N_{A_n}(Q,e_Q)$; since
$N_{A_n}(Q,e_Q) \subseteq N_{A_n}(Q)$, we may write $x = ab$ as
above. Then by Lemma \ref{alt-br}, $b$ is an even element of $S_N$,
so $a$ is an even element of $S_M$. This implies $\Aut_\F(Q)
\subseteq \Aut_{A_M}(Q)$. Conversely, conjugation by any element $a$
of $A_M$ fixes both Brauer pairs, since $a$ acts trivially on $N$,
so we also have $\Aut_{A_M}(Q) \subseteq \Aut_\F(Q)$. Hence we have
$\Aut_\F(Q) = \Aut_{A_M}(Q)$ for all $Q$.

Suppose now that $p = 2$ and we have found that $\Aut_\F(Q) =
\Aut_{S_M}(Q)$ for all $Q$. Let $P'$ be a Sylow 2-subgroup of $S_M$
containing $P$; then $P$ has index 2 in $P'$, so it is normal in
$P'$. Then $\Aut_{P'}(P)$ is a 2-subgroup of $\Aut_\F(P)$ containing
$\Aut_P(P)$, but since $\F$ is a saturated fusion system,
$\Aut_P(P)$ is a Sylow 2-subgroup of $\Aut_\F(P)$. Then we must have
$\Aut_{P'}(P) = \Aut_P(P)$. Since $P'$ properly contains $P$, we
find that $C_{P'}(P)$ must properly contain $C_P(P)$. That is, there
is an odd element $x \in P'$ that centralizes $P$.

Now let $Q$ be a centric subgroup of $P$; then $x$ centralizes $Q$.
Let $g \in S_M$ be an element representing an $\F$-automorphism of
$Q$; then $g$ and $gx$ are both elements of $S_M$ representing this
automorphism, and one of them is even. That is, any
$\F$-automorphism can be represented by an element of $A_M$, so we
have $\Aut_\F(Q) = \Aut_{A_M}(Q)$.
\end{proof}

\begin{cor}\label{alt-fusion-cor}
Let $e$ be a block idempotent of $A_n$, let $P$ be a defect group of
$e$, and let $\F$ be the block fusion system on $P$. Then there is a
subset $L$ of $\{1, \ldots, n\}$ such that $P$ is contained in $A_L$
and $\F = \F_P(A_L)$.
\end{cor}

\begin{proof}
Let $M$ be the subset of $\{1, \ldots, n\}$ consisting of those
elements that are not fixed points of $P$. By Theorem
\ref{alt-fusion}, $P$ is then a Sylow $p$-subgroup of $A_M$, and
$\F$ is equal to either $\F_P(A_M)$ or $\F_P(S_M)$. If $\F =
\F_P(A_M)$ we set $L = M$ and are done, so consider the case $\F =
\F_P(S_M)$. In this case we have $p > 2$ and $\Aut_\F(P) =
\Aut_{S_M}(P)$. Then the complement of $M$ in $\{1, \ldots, n\}$
must contain at least two elements, since otherwise we would have
$\Aut_{A_n}(P) = \Aut_{A_M}(P)$, which is impossible since
$\Aut_\F(P) \subseteq \Aut_{A_n}(P)$. Let $L$ consist of $M$ plus
any two other elements of $\{1, \ldots, n\}$; since $p > 2$, $P$ is
then a Sylow $p$-subgroup of $A_L$. By the same arguments as in the
proof of Theorem \ref{alt-fusion}, we find that $\Aut_{A_L}(Q) =
\Aut_{S_M}(Q) = \Aut_\F(Q)$ for every centric subgroup $Q$ of $P$,
so that $\F = \F_P(A_L)$.
\end{proof}

\bibliographystyle{plain}
\bibliography{Library}

\end{document}